\documentclass[a4paper,12pt]{amsart}
\usepackage[utf8]{inputenc}
\usepackage[T1]{fontenc}
\usepackage[UKenglish]{babel}
\usepackage[margin=30mm]{geometry}

\usepackage{color}%cyan,red;magenta,green;yellow,blue

\usepackage{graphicx}

\usepackage{amsmath,amssymb,amsfonts,amsthm}
\usepackage{mathrsfs,eucal,dsfont}
\usepackage{verbatim,enumitem}

\usepackage{hyperref,url}

\renewcommand{\leq}{\leqslant}
\renewcommand{\geq}{\geqslant}

\newcommand{\R}{\mathds R}
\newcommand{\Dd}{\mathds D}

\newcommand{\Pp}{\mathds P}
\newcommand{\Ee}{\mathds E}

\newcommand{\I}{\mathds 1}

\def\d{{\rm d}}
\def\<{\langle}
\def\>{\rangle}
\def\mcr{\mathscr}

\newtheorem{theorem}{Theorem}[section]
\newtheorem{lemma}[theorem]{Lemma}
\newtheorem{proposition}[theorem]{Proposition}
\newtheorem{corollary}[theorem]{Corollary}

\theoremstyle{definition}

\newtheorem{example}[theorem]{Example}
\newtheorem{remark}[theorem]{Remark}

\numberwithin{equation}{section}
\begin{document}
\allowdisplaybreaks

\title[Exponential ergodicity for nonlinear branching processes] {Exponential ergodicity for general continuous-state nonlinear branching processes}

\author{
Pei-Sen Li\qquad\,
\and\qquad
Jian Wang}
\date{}
\thanks{\emph{P.-S.\ Li:}
Institute for Mathematical Sciences, Renmin University of China, Beijing 100872, P. R. China. \url{peisenli@ruc.edu.cn}}
\thanks{\emph{J.\ Wang:}
College of Mathematics and Informatics \& Fujian Key Laboratory of Mathematical Analysis and Applications (FJKLMAA), Fujian Normal University, 350007 Fuzhou, P.R. China. \url{jianwang@fjnu.edu.cn}}

\maketitle

\begin{abstract} By using the coupling technique, we present sufficient conditions for the exponential ergodicity of general continuous-state nonlinear branching processes in both the $L^1$-Wasserstein distance and the total variation norm, where the drift term is dissipative only for large distance, and either diffusion noise or jump noise is allowed to be vanished. Sufficient conditions for the corresponding strong ergodicity are also established.

\medskip

\noindent\textbf{Keywords:} continuous-state branching process; exponential ergodicity;
coupling; strong ergodicity

\noindent \textbf{MSC 2010:} 60G51; 60G52; 60J25; 60J75.
\end{abstract}

\section{Introduction}

In this paper we will study the exponential ergodicity and the strong ergodicity for general continuous-state \emph{nonlinear branching processes}, which will be introduced below.
Consider a filtered
probability space $(\Omega, \mathscr{F}, \mathscr{F}_t, \Pp)$ satisfying the usual hypotheses. Let
$\{B_t\}_{t\ge0}$ be an $(\mathscr{F}_t)$-Brownian motion. Throughout this paper, we write $\nu$ (which is allowed to be zero) for a $\sigma$-finite nonnegative measure on $(0,\infty)$ such that $\int^\infty_0 (z\wedge
z^2) \,\nu(\d z)<\infty.$ Let
$\{N(\d s,\d z,\d u): s,z,u>0\}$ be an independent
$(\mathscr{F}_t)$-Poisson random measure on $(0,\infty)^3$ with
intensities $\d s \,\nu(\d z)\,\d u$, and $\{\tilde N(\d s,\d z,\d
u): s,z,u>0\}$ be the corresponding compensated measure, i.e.,
$\tilde N(\d s,\d z,\d
u)= N(\d s,\d z,\d
u)-\d s \,\nu(\d z)\,\d u.$ We are concerned on a general continuous-state nonlinear branching process, which is described as the pathwise unique nonnegative solution to the following stochastic differential equation (SDE):
\begin{equation}\label{sdeB}\begin{split} X_t = & X_0+\int_0^t\gamma_0(X_s)\,\d
s+\int_0^t \sqrt{\gamma_1(X_{s})} \,\d
B_s\\
&+\int_0^t\int_{0}^\infty\int_0^{\gamma_2(X_{s-})} z\,\tilde{N}(\d
s, \d z, \d u). \end{split}\end{equation}Here, \begin{itemize}

\item
$x\mapsto \gamma_0(x)$ is a continuous function  on
$\R_+:=[0,\infty)$ such that $\gamma_0(0) \ge 0$;

\item
$x \mapsto \gamma_1(x)$ is a continuous function on $\R_+$
such that $\gamma_1(0)=0$ and $\gamma_1(x)\ge 0$ for $x>0$;

\item
$x\mapsto \gamma_2(x)$ is a continuous and non-decreasing function on $\R_+$
such that $\gamma_1(0)=0$.
 \end{itemize}
Intuitively, such process can be identified as a continuous-state branching process with
population-size-dependent branching rates and with competition.

If $\gamma_0(x)=a+bx$ for some $a\ge0$ and $b\in \R$ and $\gamma_i(x)=c_i x$ $(i=1,2)$ for some $c_1,c_2\ge0$, then the solution to (\ref{sdeB}) is reduced into the classical continuous-state branching process, see \cite{BeL00, Gre74, Ky, Lam67b, Li11} and references therein. We should mention that, only and if only in this particular case, the solution satisfies the so-called branching property, which means that different individuals act independently with each other. If $ \gamma_i(x) =c_ix$ $(i=1,2 )$ for some $c_i\ge0$ and $\gamma_0(x)=b_1x- b_2x^2$ with some $b_1, b_2>0$, then the solution to  \eqref{sdeB} is called the logistic branching process in the literature, which can be used to model the population dynamics with competition, see \cite{Eh, Lam05} for more details. The quadratic regulatory term in the coefficient $\gamma_0(x)$ has an ecological interpretation, as it describes negative interactions between each pair of individuals in the population. Similar equations with general coefficients $\gamma_0(x)$ to model more general competitions were considered in \cite{Par16}.

\ \

Throughout this paper we always assume that \emph{\eqref{sdeB}  has the unique
non-explosive strong solution}, which is denoted by $(X_t)_{t\ge0}$; see Subsection \ref{section2.1} for related discussions.
Let $P_t(x,\cdot)$ and $(P_t)_{t\ge 0}$  be the transition function and the transition semigroup of the process $(X_t)_{t\ge0}$, respectively. We are going to study the asymptotic behavior of the
$L^1$-Wasserstein distance and the total variation  distance between
$P_t(x,\cdot)$ and $P_t(y,\cdot)$ for any $x,y\in\R_+$. As a direct consequence, we will establish sufficient conditions for the exponential ergodicity and the strong ergodicity of the process $(X_t)_{t\ge0}$.

To the best of our knowledge, there are few known results on this topic. For the classical branching process (i.e.\ $\gamma_0(x)=a-bx$ and $\gamma_i(x)=c_i x$ $(i=1,2)$ for some $b>0$ and $a,c_i\ge0$), by the branching property, \cite[Theorem 2.4]{LM15} proved that the total variation  distance between $P_t(x,\cdot)$ and $P_t(y,\cdot)$ decays exponentially fast. Recently, under uniformly dissipative condition on $\gamma_0(x)$ (see Remark \ref{remark3.3}(1) below) and finite second moment condition on the measure $\nu$ (i.e.\ $\int_{\R_+}z^2\,\nu(\d z)<\infty)$, \cite[Theorem 4.2]{FJKR} established the exponential decay between $P_t(x,\cdot)$ and $P_t(y,\cdot)$ with respect to the $L^1$-Wasserstein distance.

\ \

To illustrate our main contributions, we present the following statement for the exponential ergodicity and the strong ergodocity of the process $(X_t)_{t\ge0}$. The reader can refer to Section \ref{Section3} for general results.

\begin{theorem}\label{t1} Let $(X_t)_{t\ge0}$ be the unique strong solution to the SDE \eqref{sdeB} such that assumptions below \eqref{sdeB} on the coefficients are satisfied. Suppose that there are constants $l_0,k_1\ge 0$ and $k_2>0$ such that
\begin{equation}\label{e:drift00}\gamma_0(x)-\gamma_0(y)\le \begin{cases}k_1(x-y)\log \left(\frac{4l_0}{x-y}\right),\quad &0\le x-y\le l_0,\\
-k_2(x-y),\quad & x-y>l_0.
\end{cases}\end{equation}
If one of the following three assumptions holds:
\begin{itemize}
\item[$(1)$] the function $\gamma_1(x)$ is continuous and strictly positive on $(0,\infty)$ such that \begin{equation}\label{e:rem1}\liminf_{x\to 0}\frac{\gamma_1(x)}{x^\beta}>0\end{equation} for some $\beta\in [1,2)$;
 \item[$(2)$]  there are constants $\alpha\in (0,2)$ and $c_0>0$ such that
 $$\nu(\d z)\ge c_0\I_{\{0<z\le 1\}}z^{-1-\alpha}\,\d z,$$   and the function $\gamma_2(x)$ is  continuous and strictly positive on $(0,\infty)$ such that $$\liminf_{x\to 0}\frac{\gamma_2(x)}{x^\beta}>0$$ for some $\beta\in [\alpha-1,\alpha)\cap(0,\infty)$;
\item[$(3)$] there are constants $\alpha\in (1,2)$ and $c_0>0$ such that
\begin{equation}\label{e:eeefg}\int_0^r z^2\,\nu(\d z)\ge c_0 r^{2-\alpha},\quad 0<r\le 1,\end{equation} and the function  $$\gamma_2(x)=b_2 x^{r_2}+\gamma_{2,2}(x),$$ where $b_2>0$, $r_2\in [1,\alpha)$ and $\gamma_{2,2}(x)$ is a non-decreasing function on $\R_+$;
\end{itemize} then the process $(X_t)_{t\ge0}$ is exponentially ergodic both in the $W_1$-distance and the total variation  norm.

Furthermore, if \eqref{e:drift00} is replaced by \begin{equation}\label{e:drift01}\gamma_0(x)-\gamma_0(y)\le \begin{cases}k_1(x-y)\log \left(\frac{4l_0}{x-y}\right),\quad &0\le x-y\le l_0,\\
-k_2(x-y)^\delta,\quad & x-y>l_0
\end{cases}\end{equation} for some $\delta>1$, then, under one of three assumptions $(1)$--$(3)$ above, the process $(X_t)_{t\ge0}$ is strongly ergodic.
\end{theorem}

  \eqref{e:drift00} on the drift term $\gamma_0(x)$ for $0<x-y\le l_0$ is the standard one-sided non-Lipschitz continuous condition, while that for $x-y\ge l_0$ means that $\gamma_0(x)$ satisfies the dissipative condition for large distances (since $l_0$ is allowed to be any positive constant). By taking $\nu(\d z)= c_0|z|^{-1-\alpha}\I_{\{z>0\}}\,\d z$ for some $c_0>0$ and $\alpha\in (0,2)$, one can regard condition (2) as the extension of (1) from the Brownian motion case to the one-sided $\alpha$-stable noise case. When $\alpha\in (1,2)$, condition (3) on the measure $\nu$ is much weaker than condition (2); for example, \eqref{e:eeefg} is satisfied for the singular measure $\nu(\d z):=\sum_{j=0}^\infty 2^{\alpha j}\delta_{2^{-j}}(\d z)$ with $\alpha\in (1,2)$. In this case, it is at price of requiring a stronger assumption on the coefficient $\gamma_2(x)$. According to Theorem \ref{t1}, we can see that the logistic branching process (i.e.,  $ \gamma_i(x) =c_ix$ $(i=1,2 )$ for some $c_i\ge0$ and $\gamma_0(x)=b_1x- b_2x^2$ with some $b_1, b_2>0$) is strongly ergodic; see Example \ref{Exm} below for more general coefficients $\gamma_0(x)$ satisfying \eqref{e:drift01}.

In the following, we will remark that conditions \eqref{e:rem1} and \eqref{e:drift01} are sharp in some concrete examples.

\begin{remark}
(1)
Let $\gamma_0(x)=-x^2$, $\gamma_1(x)= 2x^2$ and $\gamma_2(x)=0$; that is,
$$L=x^2\frac{\d^2}{\d x^2}-x^2\frac{\d}{\d x}.$$ Define $
\mu(\d x)={x^{-2}e^{-x}}\,\d x.
$
One can verify that for any $f\in C^2_b(\R_+)$, $
\mu(Lf)=0,
$
which implies that $\mu(\d x)$ is an invariant measure for the operator $L$.
However,
$$
\mu(\R_+)=\int_0^\infty {x^{-2}e^{-x}} \,\d x\ge e^{-1}\int_{0}^1 {x^{-2}}\,\d x=\infty.
$$ On the other hand, according to \cite[the case (i)-(ib) after Example 2.18, p.\ 14]{LYZ}, we know that $\Pp^x(\tau_0=\infty)=1$ for all $x>0$, where $\tau_0=\inf\{t>0: X_t=0\}$. This is, the point $0$ can be seen as the reflection boundary for the diffusion process $(X_t)_{t\ge0}$ associated with the operator $L$ on $[0,\infty)$.
Therefore, the process $(X_t)_{t\ge0}$ is not ergodic, see e.g. \cite[Table 5.1, p.\ 100]{Chen}.
Note that, for this example,  (\ref{e:rem1}) is satisfied with $\beta=2$, and so this implies that \eqref{e:rem1} with $\beta<2$ in Theorem \ref{t1} is optimal.

(2) Let $\gamma_0(x)=d-bx$ with $b, d>0$, $\gamma_1(x)=\sqrt{2c}x$ with $c>0$ and $\gamma_2(x)=0$. Then, the solution to (\ref{sdeB}) is reduced into the famous Cox-Ingersoll-Ross (CIR) model. In this case, one can easily see that (\ref{e:drift00}) and (1) in Theorem \ref{t1} hold. Therefore, the CIR model is exponentially ergodic in both the $W_1$-distance and the total variation  distance.
On the other hand, denote by $\tau_1=\inf\{t\ge0: X_t=1\}$. According to  \cite[Corollary 9]{DuFoMa14},  for any $x>1$,
$$
\Ee^x[\tau_1]=\int^\infty_0\frac{e^{-z}-e^{-xz}}{bz+cz^2}\exp\left(\int^z_0\frac{d}{bu+cu^2}\,\d u\right)\,\d z.
$$
 By letting $x\to\infty$ in the above equality, we can conclude that $\sup_{x>1}\Ee^x[\tau_1]=\infty$. This together with \cite[Lemma 2.1]{Mao02} yields that the CIR model is not strongly ergodic. In particular, this implies that \eqref{e:drift01} with $\delta>1$ for the strong ergodicity in some sense is sharp.\end{remark}

The approach of our paper is based on recent developments of the couplings for SDEs with L\'evy noises via coupling operators, see \cite{LW, LSW, LW19, LW18, W} for more details. However, there are a few essential differences between continuous-state nonlinear branching processes and the settings of \cite{LW, LSW, LW19, LW18, W}. For example, in the present setting, the diffusion term and the jump noise are allowed to appear together in the SDE \eqref{sdeB}, and moreover both coefficients $\gamma_1(x)$ and $\gamma_2(x)$ are degenerate on $\R_+$ (since $\gamma_1(0)=\gamma_2(0)=0$). The differences require much more effort than those as in  \cite{LW, LSW, LW19, LW18, W} to efficiently apply the coupling technique. In particular, we need consider the coupling operator that contains both local part and non-local part of the associated generator \eqref{opeB}, which makes the coupling function (e.g., see \eqref{lem-test-funct.1} and \eqref{lem-test-funct.222}) in the applications of coupling process more complex and delicate than that in \cite{LW, LSW, LW19, LW18, W}.

\ \

The remainder of this paper is arranged as follows. In Section \ref{sec2}, we recall some results from \cite{FL10} on the strong solution to the SDE \eqref{sdeB}, and then present a Markovian coupling of the solution through the construction of a new coupling operator. General results on the exponential ergodiciy and the strong ergodicity for the SDE \eqref{sdeB} are stated in  Section \ref{Section3}. The proofs of all main results in Section \ref{Section3} and Theorem \ref{t1} are given in the last section.

\section{Unique strong solution and its coupling process}\label{sec2}
This section consists of two parts. We first recall results from \cite{FL10} on the existence and the uniqueness of the strong solution to the SDE \eqref{sdeB}, and then construct a new Markovian coupling of the solution.

\subsection{Existence and uniqueness of strong solution}\label{section2.1}
The statement is  taken form \cite[Theorem 5.6]{FL10}.

\begin{theorem}\label{T:solu} {\rm (\cite[Theorem 5.6]{FL10})}  Suppose that the coefficients $\gamma_i$, $i=0,1,2,$ satisfy the following conditions:
 \begin{itemize}
\item[{\rm (1)}] there is a constant $K>0$ so that
$$\gamma_0(x)\le K(1+x),\quad x\ge0;$$
\item[{\rm (2)}]  there exists a non-decreasing function $L(x)$ on $\R_+$ such that $$\gamma_1(x)\le L(x),\quad x\ge0;$$

\item[{\rm (3)}] the function $\gamma_2(x)$ is nonnegative and non-decreasing on $\R_+;$

\item[{\rm (4)}] $\gamma_0(x)= \gamma_{0,1}(x) - \gamma_{0,2}(x)$, where $\gamma_{0,1}(x)$ is continuous on $\R_+$, and $\gamma_{0,2}(x)$ is continuous and
non-decreasing on $\R_+$. For each integer $m\ge 1$ there is a non-decreasing concave function
$r_m(x)$ on $\R_+$ such that $\int_{0}^1 r_m(z)^{-1}\, \d z
= \infty$, and for all $0\le x,y\le m$,
$$
|\gamma_{0,1}(x)-\gamma_{0,1}(y)| \le r_m(|x-y|);
$$

\item[{\rm(5)}] for each integer $m\ge1$ there is a nonnegative and non-decreasing  function  $\rho_m(x)$ on $\R_+$ such that $\int_{0}^1 \rho_m(z)^{-2}\, \d z
= \infty$, and for all $0\le x,y\le m$,
$$
|\sqrt{\gamma_1(x)}-\sqrt{\gamma_1(y)}|^2+ |\gamma_2(x) - \gamma_2(y) |\le  \rho_m(|x-y|)^2 .
$$
\end{itemize}
Then, for any initial value $X_0=x\ge0$, there exists
a unique strong solution to the SDE {\rm(\ref{sdeB})}, and  the solution is a strong Markov process $(X_t)_{t\ge0}$ with the generator given by \begin{equation}\label{opeB} Lf(x) =
\gamma_0(x)f^\prime(x)+\frac{\gamma_1(x)}{2}f^{\prime\prime}(x) +
\gamma_2(x)\int_0^\infty \big(f(x+z) - f(x) -
zf^\prime(x)\big)\,\nu(\d z) \end{equation} for any $f\in C_b^2(\R_+)$.
\end{theorem}

To investigate the exponential ergodicity of the process $(X_t)_{t\ge0}$,  we will assume that the drift term $\gamma_0(x)$ is dissipative for large
distance, see \eqref{e:drift} below. One can see that condition (1) in Theorem \ref{T:solu} holds with $K=\sup_{0\le r\le l_0}\Phi_1(r)$ under \eqref{e:drift}. On the other hand, we suppose that the function $\gamma_1(x)$ is continuous on $\R_+$
such that $\gamma_1(0)=0$. Hence, condition (2) in Theorem \ref{T:solu} holds with $L(x):=\sup_{0\le y\le x} \gamma_1(y)$. We have already supposed that condition (3) is satisfied, see assumptions below the SDE \eqref{sdeB}.
Therefore, in the setting of our paper,  the SDE \eqref{sdeB} has a unique strong solution under assumptions of Theorem \ref{Th:W} (or Theorem \ref{Th:W1}), and some locally continuous assumptions on the coefficients $\gamma_i(x)$ for all $i=0,1,2$ (e.g.\  conditions (4) and (5) in Theorem \ref{T:solu}).

\subsection{Markovian coupling for continuous-state nonlinear branching process}
To study the coupling of the process $(X_t)_{t\ge0}$ determined by \eqref{sdeB}, we  begin with the construction of a new coupling operator for its generator $L$ given by \eqref{opeB}.
Recall that an operator $\tilde{L}$ acting on $C^2(\R^2_+)$ is called a \emph{coupling} of $L$ given by \eqref{opeB}, if for any $f, g\in C^2(\R_+)$,
$$
\tilde{L}h(x,y)=Lf(x)+Lg(y),
$$ where $h(x,y)=f(x)+g(x)$ for $x,y\in\R_+$.
In this paper, we will use the coupling by reflection
of the local part and the refined basic coupling of the non-local part
for the operator $L$. Note that the function $\gamma_2(x)$ is non-decreasing on $\R_+$. For a given parameter
$\kappa>0$, set $x_{\kappa}=x\wedge\kappa$ for $x>0$. Roughly speaking, when
$x>y\ge0$, the refined basic coupling of the non-local part for the operator $L$ is given by
$$
(x,y)\longrightarrow\begin{cases}(x+z,y+z+(x-y)_\kappa),\quad& \frac{1}{2}\gamma_2(y)\mu_{-(x-y)_{\kappa}}(\d z),\\
(x+z,y+z- (x-y)_{\kappa}), \quad & \frac{1}{2}\gamma_2(y)\mu_{(x-y)_{\kappa}}(\d z),\\
(x+z, y+z), \quad & \gamma_2(y)\Big[\nu(\d z)-\frac{1}{2}\mu_{-(x-y)_{\kappa}}(\d z)-\frac{1}{2}\mu_{(x-y)_{\kappa}}(\d z)\Big],\\
(x+z, y), \quad &[\gamma_2(x)-\gamma_2(y)]\nu(\d z),
\end{cases}
$$ where \begin{equation}\label{e:measure0}\mu_x=\nu\wedge(\delta_x*\nu)(\d z)\end{equation} for all $x\in \R$. Similarly, we can define the case that $0\le x<y$. See \cite[Section 2]{LW18} for more details on the refined basic coupling for SDEs with L\'evy jumps.
Then,  for any $f\in C^2(\R_+^2)$ and $x>y\ge0$, we define
\begin{equation}\label{generator}\begin{split}
\tilde{L}f(x,y) =&\gamma_0(x)f'_x(x,y)+\gamma_0(y)f'_y(x,y) \\
&+\frac{1}{2}\gamma_1(x)f''_{xx}(x,y)+\frac{1}{2}\gamma_1(y)f''_{yy}(x,y)-\sqrt{\gamma_1(x)\gamma_1(y)}f''_{xy}(x,y)\\
&+\frac{1}{2}\gamma_2(y)\int^\infty_0(f(x+z,y+z+(x-y)_\kappa)-f(x,y)\\
&\qquad\qquad\qquad\,\,-f_x'(x,y)z-f_y'(x,y)(z+(x-y)_\kappa))\,\mu_{-(x-y)_{\kappa}}(\d z)\\
&+\frac{1}{2}\gamma_2(y)\int^\infty_0(f(x+z,y+z-(x-y)_\kappa)-f(x,y)\\
&\qquad\qquad\qquad\,\,-f_x'(x,y)z-f_y'(x,y)(z-(x-y)_\kappa))\,\mu_{(x-y)_{\kappa}}(\d z)\\
&+ \gamma_2(y)\int^\infty_0(f(x+z,y+z)-f(x,y)-f_x'(x,y)z\\
&\qquad\qquad\quad\,\,\, \,-f_y'(x,y)z)\,\Big(\nu -\frac{1}{2}\mu_{-(x-y)_{\kappa}} -\frac{1}{2}\mu_{(x-y)_{\kappa}}\Big)(\d z) \\
&+(\gamma_2(x)-\gamma_2(y))\int^\infty_0(f(x+z,y)-f(x,y)-f_x'(x,y)z)\,\nu(\d z).
\end{split}
\end{equation} Here and in what follows, $f'_x(x,y)=\frac{\partial f(x,y)}{\partial x}$, $f''_{xx}(x,y)=\frac{\partial^2 f(x,y)}{\partial x^2}$ and $f''_{xy}(x,y)=\frac{\partial^2 f(x,y)}{\partial x\partial y}$, and so on. Similarly, we can define $\tilde L f(x,y)$ for the
case that $0\le x<y$. By using the fact that
$\mu_x=\delta_x*\mu_{-x}$ for any $x\in \R$ (see \cite[Corollary A.2]{LW18}), one can check that the generator $\tilde{L}$ constructed above is a coupling operator of $L$ given by \eqref{opeB}; see \cite[Subsection 2.1]{LW18}.

Next, we will construct the SDE on $\R_+^2$ associated with the
coupling operator $\tilde L$ defined above, and  prove the existence
of the strong solution to the corresponding SDE.
The idea below is partly motivated by \cite[Subsection 2.2]{LW18}.
According to \cite[Corollary A.2 and Remark 2.1]{LW18},
$\mu_x=\delta_x*\mu_{-x}$, and \begin{equation}\label{e:note}\mu_x(\R_+)=\mu_{-x}(\R_+)\le
2\nu(\{z\in \R_+: z>|x|/2\})<\infty\end{equation} for all $x\in \R$. Recalling $\mu_x=\nu\wedge(\delta_x*\nu)(\d z)$, we define the following control function
$$
\rho(x,z)=\frac{\mu_x(\d z)}{\nu(\d z)}\in[0,1],\quad x\in \R,\,z\in \R_+
$$
with $\rho(0,z)=1$ by convention. Consider the following SDE:
\begin{equation}\label{sdeD}\begin{split}\begin{cases} X_t =  & x+\displaystyle\int_0^t\gamma_0(X_s)\,\d
s+\displaystyle\int_0^t \sqrt{\gamma_1(X_{s})} \,\d
B_s\\
&+\displaystyle\int_0^t\int_{0}^\infty\int_0^{\gamma_2(X_{s-})} z\,\tilde{N}(\d
s, \d z, \d u)\\
 Y_t = & y+\displaystyle\int_0^t\gamma_0(Y_s)\,\d
s-\displaystyle\int_0^t \sqrt{\gamma_1(Y_{s})} \,\d
B^*_s\\
&+\displaystyle\int_0^t\int_{0}^\infty\int_0^{\frac{1}{2}\gamma_2(Y_{s-})\rho(-(U_{s-})_\kappa,
z)} [z+(U_{s-})_{\kappa}]\,\tilde{N}(\d
s, \d z, \d u)\\
&+\displaystyle\int_0^t\int_{0}^\infty\int^{\frac{1}{2}\gamma_2(Y_{s-})[\rho(-(U_{s-})_\kappa,
z)+\rho((U_{s-})_\kappa,
z)]}_{\frac{1}{2}\gamma_2(Y_{s-})\rho(-(U_{s-})_\kappa, z)}
[z-(U_{s-})_{\kappa}]\,\tilde{N}(\d
s, \d z, \d u)\\
&+\displaystyle\int_0^t\int_{0}^\infty\int_{\frac{1}{2}\gamma_2(Y_{s-})[\rho(-(U_{s-})_\kappa,
z)+\rho((U_{s-})_\kappa, z)]}^{\gamma_2(Y_{s-})} z\,\tilde{N}(\d s,
\d z, \d u),\end{cases}
\end{split}\end{equation} where $$
B^*_t=\begin{cases}-B_t,\quad& t\le {T},\\
-2B_{{T}}+B_{t}, \quad & t>{T},
\end{cases}
$$
${T}=\inf\{t\ge0: X_t=Y_t\}$, and $U_t=X_t-Y_t$.

\begin{proposition}
For any $(x,y)\in \R_+^2$, the system of equations \eqref{sdeD} is
well defined, and has a unique strong solution $(X_t, Y_t)_{t\ge 0}$. Moreover, we have
\begin{itemize}
\item[{\rm(1)}] the infinitesimal generator of the process $(X_t, Y_t)_{t\ge 0}$
is just the coupling operator $\tilde{L}$ defined by \eqref{generator}.
\item[{\rm(2)}] $X_t=Y_t$ for all $t\ge T$, where
$T=\inf\{t>0:X_t=Y_t\}.$\end{itemize}
\end{proposition}
\begin{proof}
Recall that in the setting of our paper, we always assume that
(\ref{sdeB}) has a non-explosive and pathwise unique strong solution
$(X_t)_{t\ge 0}$. We are going to show that the sample paths of
$(Y_t)_{t\ge 0}$ given in \eqref{sdeD} can be obtained by repeatedly modifying those of
the solution of the following equation:
\begin{equation}\label{sdeD0}\begin{split}
 Z_t = & y+\int_0^t\gamma_0(Z_s)\,\d
s-\int_0^t \sqrt{\gamma_1(Z_{s})} \,\d
B^*_s+\int_0^t\int_{0}^\infty\int_0^{\gamma_2(Z_{s-})} z\,\tilde{N}(\d
s, \d z, \d u).
\end{split}\end{equation}  By the definition of $(B^*_t)_{t\ge0}$, we
can verify that $(B_t^*)_{t\ge0}$ is still an
$\mathscr{F}_t$-Brownian motion. Since the driving Poisson random measure for \eqref{sdeB} and \eqref{sdeD0} is same, the existence of the strong solution $(Z_t)_{t\ge0}$ to the equation (\ref{sdeD0}) is guaranteed by the pathwise unique strong solution to \eqref{sdeB}.

We first claim that the process $(Y_t)_{t\ge0}$ given in \eqref{sdeD} is the same as
\begin{equation}\label{e:ooofff}\begin{split} Y_t = & y+\displaystyle\int_0^t\gamma_0(Y_s)\,\d
s-\displaystyle\int_0^t \sqrt{\gamma_1(Y_{s})} \,\d
B^*_s\\
&+\int_0^t\int_{0}^\infty\int_0^{\gamma_2(Y_{s-})} z\,\tilde{N}(\d
s, \d z, \d u)\\
&+\displaystyle\int_0^t (U_{s-})_{\kappa}\int_{0}^\infty\int_0^{\frac{1}{2}\gamma_2(Y_{s-})\rho(-(U_{s-})_\kappa,
z)}\,{N}(\d
s, \d z, \d u)\\
&-\displaystyle\int_0^t (U_{s-})_{\kappa}\int_{0}^\infty\int^{\frac{1}{2}\gamma_2(Y_{s-})[\rho(-(U_{t-})_\kappa,
z)+\rho((U_{t-})_\kappa,
z)]}_{\frac{1}{2}\gamma_2(Y_{s-})\rho(-(U_{s-})_\kappa, z)}
\,{N}(\d
s, \d z, \d u).\end{split}\end{equation} Indeed, this immediately follows from the fact that for all $z>0$, $\mu_z(\R_+)=\mu_{-z}(\R_+)<\infty$, and the identity that for any $x,y\in \R_+$ with $x\neq y$,
\begin{align*}\int_0^\infty \int_0^{\frac{1}{2}\gamma_2(y)\rho(-(x-y)_\kappa,
z)}\,\d u \,\nu(\d z)=&\frac{1}{2}\gamma_2(y)\mu_{-(x-y)_\kappa}(\R_+)=\frac{1}{2}\gamma_2(y)\mu_{(x-y)_\kappa}(\R_+)\\
=&\int_0^\infty\int_{\frac{1}{2}\gamma_2(y)\rho(-(x-y)_\kappa,
z)}^{\frac{1}{2}\gamma_2(y)[\rho(-(x-y)_\kappa,
z)+\rho((x-y)_\kappa
z)]}\,\d u \,\nu(\d z).\end{align*} Hence, we next turn to construct the sample paths of
$(Y_t)_{t\ge 0}$ given in \eqref{e:ooofff}.

Let $(Z_t^{(1)})_{t\ge 0}$ be the solution to (\ref{sdeD0}) with $Z_0^{(1)}=y$. Denote
by $(p_t)_{t\in \Dd_p}$ the Poisson point process associated with the Poisson random measure $N(\d
s, \d z, \d u)$ on $(0, \infty)^2$, and by $\Delta X_t=X_t-X_{t-}$. Let $R^{(1)}_t=R^{(1)}_{1,t}+R^{(1)}_{2,t},$
where $$
R^{(1)}_{1,t}:=\frac{1}{2}\gamma_2(Z^{(1)}_{t}) \rho(-(X_{t}-Z^{(1)}_{t})_\kappa,
\Delta X_t),\quad R^{(1)}_{2,t}:=\frac{1}{2}\gamma_2(Z^{(1)}_{t})\rho((X_{t}-Z^{(1)}_{t})_\kappa, \Delta X_t).
$$
Define the stopping times $T_1=\inf\{t>0: Z^{(1)}_t=X_t\}$, and
$$
\sigma_1=\inf\left\{t\in \Dd_p:
p_t\in(0,\infty)\times(0,R_t^{(1)}]\right\}.
$$ 
We consider two cases:

(i) On the event $\{T_1\le\sigma_1\},$ we set $Y_t=Z^{(1)}_t$ for
all $t< T_1$; moreover, by the pathwise uniqueness of (\ref{sdeB}),
we can define $Y_t=X_t$ for $t\ge T_1.$

(ii) On the event $\{T_1> \sigma_1\}$, we define $Y_t=Z^{(1)}_t$ for
all $t<\sigma_1$ and
$$
Y_{\sigma_1}=Z^{(1)}_{\sigma_1-}+\Delta X_{\sigma_1}
+\begin{cases}(X_{\sigma_1-}-Y_{\sigma_1-})_\kappa,\qquad& p_{\sigma_1}\in (0,\infty)\times(0, R^{(1)}_{1,t}],\\
-(X_{\sigma_1-}-Y_{\sigma_1-})_\kappa, \qquad &
p_{\sigma_1}\in (0,\infty)\times(R^{(1)}_{1,t},
R^{(1)}_t].
\end{cases}
$$

In the following, we will restrict on the event $\{T_1> \sigma_1\}$ and consider the
following SDE: \begin{equation*}\begin{split}
 Z_t = & Y_{\sigma_1}+\int_{\sigma_1}^t\gamma_0(Z_s)\,\d
s-\int_{\sigma_1}^t \sqrt{\gamma_1(Z_{s })} \,\d
B^*_s\\
&+\int_{\sigma_1}^t\int_0^\infty\int_0^{\gamma_2(Z_{s-})}
z\,\tilde{N}(\d s, \d z, \d u),\qquad t>\sigma_1.
\end{split}\end{equation*}
Denote its solution by $(Z^{(2)}_t)_{t\ge \sigma_1}$. Similarly, we
set $R^{(2)}_t=R^{(2)}_{1,t}+R^{(2)}_{2,t}$ with
$$
R^{(2)}_{1,t}:=\frac{1}{2}\gamma_2(Z^{(2)}_{t-})\rho(-(X_{t-}-Z^{(2)}_{t-})_\kappa,
\Delta X_t),\quad R^{(2)}_{2,t}:=\frac{1}{2}\gamma_2(Z^{(2)}_{t-})\rho((X_{t-}-Z^{(2)}_{t-})_\kappa, \Delta X_t)]$$ for all $
t>\sigma_1.
$ We further define
$T_2=\inf\{t>0: Z^{(2)}_t=X_t\}$, and
$$
\sigma_2=\inf\left\{t\in \Dd_p\cap (\sigma_1,\infty):
p_t\in(0,\infty)\times(0, R^{(2)}_t]\right\}.
$$ 
In the same way, we can define $Y_t$ for $t\le\sigma_2.$ We then repeat
this procedure. Note that
$$
\frac{1}{2}\gamma_2(Y_{t-})[\mu_{-(X_{t-}-Y_{t-})_\kappa(\R_+)}+\mu_{(X_{t-}-Y_{t-})_\kappa(\R_+)}]
$$
is uniformly bounded (thanks to \eqref{e:note}) for any $t<\tau_m$ with $m=1,2,\dots$, where $$\tau_m=\inf\{t\ge
0: Y_t>m~\mbox{or}~ |X_t-Y_t|<1/m\}.$$ Then only finite many
modifications have to be made in the interval $(0, t\wedge \tau_m)$.
Finally, by letting $m\to\infty$, we can determine the unique strong
solution $(Y_t)_{t\ge 0}$ to the SDE \eqref{e:ooofff} globally.

With the construction of $(Y_t)_{t\ge0}$ above, we can apply the It\^{o} formula to the SDE \eqref{sdeD} to obtain the assertion (1). The assertion (2) immediately follows from the SDE \eqref{sdeD} and the assumption that
(\ref{sdeB}) has a non-explosive and pathwise unique strong solution
$(X_t)_{t\ge 0}$.
\end{proof}

In the following, we call $(X_t,Y_t)_{t\ge0}$ determined by \eqref{sdeD} a (Markovian) coupling process of  $(X_t)_{t\ge0}$.
To conclude this part, we will give the preserving order property of the coupling process $(X_t,Y_t)_{t\ge0}$.

\begin{corollary}\label{L:order}Let $(X_t,Y_t)_{t\ge0}$ be the coupling process determined by \eqref{sdeD} and with the starting point $(x,y)$. If $x>y$, then $X_t\ge Y_t$ for all $t>0$ a.s.\end{corollary}
\begin{proof} Denote by $\Pp^{(x,y)}$ and $\Ee^{(x,y)}$ the probability and the expectation of the process $(X_t,Y_t)_{t\ge0}$ starting from $(x,y)$, respectively.
Let $$\tilde T:=\inf\{t>0: Y_t>X_t\},$$ and define  $
f_n\in C_b^2(\R_+^2)$ such that $f_n\ge0$,  $f_n(x,y)=1$ if $y\ge x+1/n$, and
$f_n(x,y)=0$ if $y<x$. Then, for any $x>y$ and $t>0$,
\begin{align*}\Ee^{(x,y)}f_n(X_{t\wedge \tilde T},Y_{t\wedge \tilde T})=&
f_n(x,y)+\Ee^{(x,y)} \left(\int_0^{t\wedge \tilde T} \tilde
Lf_n(X_s,Y_s)\,\d s\right)=0,\end{align*} where in the last equality
we used the fact that $\tilde L f_n(x,y)=0$ for all $x\ge y$, thanks
to the definition of the coupling operator $\tilde L$ given by \eqref{generator} and the assumption that the function $\gamma_2(x)$ is non-decreasing on $\R_+$. Then, by the
Fatou lemma,
$$\Pp^{(x,y)}(\tilde T<t)=\Ee^{(x,y)}\liminf_{n\to\infty}f_n(X_{t\wedge \tilde T},Y_{t\wedge
\tilde T})\le \liminf_{n\to \infty}\Ee^{(x,y)}  f_n(X_{t\wedge
\tilde T},Y_{t\wedge \tilde T})=0.$$ Therefore, for any $x>y$,
$$\Pp^{(x,y)}(\tilde T=\infty)=1.$$ That is, for any $x>y$, the coupling
process $(X_t,Y_t)_{t\ge0}$ associated with the coupling operator
$\tilde L$ satisfies that $X_t\ge Y_t$ for all $t>0$ a.s.
\end{proof}

\begin{remark}\label{remark2.4} One can apply the synchronous coupling (instead of the refined basic coupling) of the non-local part
to construct another coupling operator for the operator $L$. For any $x>y\ge0,$ the synchronous
coupling of the non-local part for the operator $L$ given by \eqref{opeB} is given by
$$(x,y)\longrightarrow\begin{cases}(x+z,y+z),\quad& \gamma_2(y)\,\nu(\d z),\\
(x+z,y), \quad &(\gamma_2(x)-\gamma_2(y))\,\nu(\d z).\end{cases}$$
Then, for any $f\in C^2(\R_+^2)$ and $x>y\ge0$, the corresponding coupling operator $L^*$ is defined by
\begin{align*}
L^*f(x,y) =&\gamma_0(x)f'_x(x,y)+\gamma_0(y)f'_y(x,y)(x,y)\\
&+\frac{1}{2}\gamma_1(x)f''_{xx}(x,y)+\frac{1}{2}\gamma_1(y)f''_{yy}(x,y)-\sqrt{\gamma_1(x)\gamma_1(y)}f''_{xy}(x,y)\\
&+\gamma_2(y)\int^\infty_0(f(x+z,y+z)-f(x,y)-f_x'(x,y)z-f_y'(x,y)z)\,\nu(\d z)\\
&+(\gamma_2(x)-\gamma_2(y))\int^\infty_0(f(x+z,y)-f(x,y)-f_x'(x,y)z)\,\nu(\d z).
\end{align*} The difference between $\tilde L$ and $L^*$ is that the coupling operator  $L^*$ do not involve the measures $\mu_{(x-y)_\kappa}$ and $\mu_{-(x-y)_\kappa}$. The coupling process associated with the coupling operator $L^*$ above can be construed directly.  Actually, putting \eqref{sdeB} and \eqref{sdeD0} together, we can
check by It\^{o}'s formula that the generator of the Markov process
$(X_t, Y_t)_{t\ge0}$ defined by \eqref{sdeB} and \eqref{sdeD0} on $\R_+^2$ is just the coupling operator
$L^*$; moreover, $X_t=Y_t$ for $t\ge{T}$, where
$T=\inf\{t>0:X_t=Y_t\}.$ Similarly, we can see that this coupling process $(X_t,Y_t)_{t\ge0}$ also enjoys the preserving order property as in Corollary \ref{L:order}.
\end{remark}

\section{Exponential convergence in the $L^1$-Wasserstein distance and the total variation distance}\label{Section3}

In this section, we shall give general results about the exponential
ergodicity of the process $(X_t)_{t\ge0}$ determined by the SDE
\eqref{sdeB}, in terms of both the $L^1$-Wasserstein distance and the total variation norm.
To present our main result, we first introduce some notation. For a strictly increasing  function $\psi$ on $\R_+$ and two
probability measures $\mu_1$ and $\mu_2$ on $\R_+$, define $$
W_{\psi}(\mu_1,\mu_2)=\inf_{\Pi\in\mathscr{C}(\mu_1,
\mu_2)}\int_{\R_+^2} \psi(|x-y|)\,\Pi(\d x, \d y), $$ where
$\mathscr{C}(\mu_1, \mu_2)$ is the collection of measures on
$\R_+^2$ with marginals $\mu_1$ and $\mu_2$. When $\psi$ is concave,
the above definition gives rise to a Wasserstein distance $W_\psi$
in the space of probability measures $\mu$ on
$\R_+$ such that $\int_{\R_+}\psi(z)\,\mu(\d z)<\infty$. If
$\psi(r)=r$ for all $r\geq0$, then $W_\psi$ is the standard
$L^1$-Wasserstein distance, which will be denoted by $W_1$
throughout this paper. Another well known example for $W_\psi$ is
given by $\psi(r)=\I_{(0,\infty)}(r)$, which leads to the total
variation distance $$W_\psi(\mu_1,
\mu_2)=\frac{1}{2}\|\mu_1-\mu_2\|_{\rm Var}:=\frac{1}{2}
[(\mu_1-\mu_2)^+(\R_+)+(\mu_1-\mu_2)^-(\R_+)].$$

\ \

The following two results give us the exponential convergence in the $L^1$-Wasserstein distance and the total variation norm for the SDE (\ref{sdeB}), respectively.

\begin{theorem}\label{Th:W}
 Suppose that  there are constants $l_0\ge 0$, $k_2>0$ and a nonnegative function $\Phi_1\in C[0, 2l_0]\cap C^3(0, 2l_0]$ satisfying $\Phi_1(0)=0$, $\Phi'_1\ge0$, $\Phi_1''\le0$ and $\Phi_1'''\ge0$ on $(0,2l_0]$ such that
\begin{equation}\label{e:drift}\gamma_0(x)-\gamma_0(y)\le \begin{cases}\Phi_1(x-y),\quad &0\le x-y\le l_0,\\
-k_2(x-y),\quad & x-y>l_0.
\end{cases}\end{equation}
If one of the following two assumptions holds:
 \begin{itemize}
\item[{\rm(A1)}] there exist constants $\beta\in[1,2)$ and $k_3>0$ such that
$$
\int_{0}^1\Phi_1(r)r^{-\beta}\,\d r<\infty
$$ and \begin{equation}\label{e:diff} {\gamma_1(x)}+ {\gamma_1(y)} \ge k_3(x-y)^{\beta},\quad 0\le x-y\le l_0;
\end{equation}
\item[{\rm(A2)}] there exist constants $\alpha\in (0,2)$,  $\beta\in[\alpha-1,\alpha)\cap(0,\infty)$ and $C_*, k_3>0$  such that
$$
\int_{0}^1\Phi_1(r)r^{\alpha-\beta-2}\,\d r<\infty,
$$ \begin{equation}\label{e:jump01}\int_0^r z^2\,\nu(\d z)\ge C_* r^{2-\alpha},\quad 0<r\le 1\end{equation}
 and
\begin{equation}\label{e:jump}(\gamma_2(x)-\gamma_2(y)) +\gamma_2(y)\I_{\{\inf_{0<z\le \kappa}[z^{\alpha}
\mu_z(\R_+)]\ge C_*\}} \ge k_3(x-y)^{\beta},\quad 0\le x-y\le l_0,
\end{equation} where $\mu_z$ is given by \eqref{e:measure0};
 \end{itemize}
then there exist positive constants $C$ and $\lambda$ so that for all $t>0$ and $x,y\ge0$,
$$W_1(P_t(x,\cdot),P_t(y,\cdot))\le C e^{-\lambda t} |x-y|.$$
\end{theorem}

\begin{theorem}\label{Th:W1} Under the assumptions of Theorem $\ref{Th:W}$, if additionally the function $\Phi_1$ in \eqref{e:drift} satisfies $$\limsup_{r\to0} \Phi_1(r)r^{1-\beta}=0$$
 when Assumption {\rm (A1)} holds, or satisfies $$\limsup_{r\to0} \Phi_1(r)r^{\alpha-\beta-1}=0$$
 when Assumption {\rm (A2)} holds, then there exist positive constants $C$ and $\lambda$ so that for all $t>0$ and $x,y\ge0$,
$$\|P_t(x,\cdot)-P_t(y,\cdot)\|_{\rm Var}\le C e^{-\lambda t}(1+ |x-y|).$$ \end{theorem}

We make some comments on the assumptions of Theorems \ref{Th:W} and \ref{Th:W1}. First, \eqref{e:drift} is the so-called dissipative condition for large distance on the drift term $\gamma_0(x)$. In applications there are a lot of choices for the function $\Phi_1$; for example, $\Phi_1(r)=Cr$ corresponds to the standard one-sided locally Lipschitz continuous condition, and  $\Phi_1(r)=C r\log(4l_0/r)$ is the typical one-sided non-Lipschitz continuous condition. Both functions satisfy assumptions in Theorems \ref{Th:W} and \ref{Th:W1}. Secondly, since we assume that the function $\gamma_1(x)$ is continuous on $\R_+$ such that $\gamma_1(0)=0$, \eqref{e:diff} is satisfied when the function $\gamma_1(x)$ is strictly positive on $(0,\infty)$ such that $\liminf_{x\to0} \frac{\gamma_1(x)}{x^\beta} >0$. Thirdly, \eqref{e:jump01} implies that $\int_0^1\,\nu(\d z)=\infty$. Suppose furthermore $\mu_x(\R_+)\ge C_* x^{-\alpha}$ for all $x\in (0,\kappa]$.  This assumption is concerned on  the concentration of the L\'evy measure $\nu$ around zero (small jump activity), and it implies that the measure $\nu$ has a component that is absolutely continuous with respect to the Lebesgue measure, see \cite[Proposition A.5]{LW18}. Then \eqref{e:jump} is equivalently saying that $\gamma_2(x)\ge k_3 x^\beta$ for all $0\le x\le l_0$, which is also equivalent that $\gamma_2(x)$ is strictly positive on $(0,\infty)$ such that $\liminf_{x\to0} \frac{\gamma_2(x)}{x^\beta} >0$. On the other hand, when $\gamma_2(x)-\gamma_2(y)\ge k_3(x-y)^\beta$ for all $0<x-y\le l_0$ (this in particular indicates that the function $\gamma_2$ is strictly increasing on $\R_+$), we only require \eqref{e:jump}, which can be fulfilled even for singular measures $\nu$, see the remarks below Theorem \ref{t1}.

\ \

As direct consequences of Theorems \ref{Th:W} and \ref{Th:W1}, we have the following
statement for the exponential ergodicity of the process
$(X_t)_{t\ge0}$ in term of the $W_1$-distance and the total variation norm. Let $\mcr{P}_1$ be the space of probability measures having the first finite moment.

\begin{corollary}\label{Cor:W}\begin{itemize}
\item[$(1)$] Under assumptions of Theorem $\ref{Th:W}$,  there exist a unique
invariant probability measure $\mu\in \mcr{P}_1$ and a constant $\lambda>0$ such that for all $t>0$ and
$\mu_0\in\mcr{P}_1$,
$$W_1(\mu_0 P_t,\mu)\le C_{\mu_0} e^{-\lambda t},$$ where $C_{\mu_0}$ is a positive constant depending on $\mu_0$.

\item[$(2)$] Under assumptions of Theorem $\ref{Th:W1}$,  there exist a unique
invariant probability measure $\mu\in \mcr{P}_1$ and a constant $\lambda>0$ such that for all $t>0$ and
$\mu_0\in\mcr{P}_1$,  $$\|\mu_0 P_t-\mu\|_{\rm Var}\le C_{\mu_0}e^{-\lambda t},$$ where $C_{\mu_0}$ is a positive constant depending on $\mu_0$. \end{itemize}  \end{corollary}

\begin{remark} \label{remark3.3}
(1)  Recently, under the uniformly dissipative condition on the drift
term $\gamma_0(x)$, i.e., \eqref{e:drift} holds with $l_0=0$, which is equivalently saying that
$$\gamma_0(x)-\gamma_0(x)\le -k_2(x-y),\quad 0\le y\le x,$$ and the finite second moment condition for the jump
measure $\nu$, as well as some growth conditions on the coefficients $\gamma_1(x)$ and $\gamma_2(x)$, \cite[Theorem 4.2]{FJKR} establishes the
exponential ergodicity in the $L^1$-Wasserstein distance for
continuous-state nonlinear branching processes. Clearly, Theorem
\ref{Th:W} and Corollary \ref{Cor:W}(1) improve \cite[Theorem
4.2]{FJKR}. Here the drift term is only required to be dissipative for large
distances as indicated by \eqref{e:drift}, or the jump measure with finite first moment. In
particular, Theorem \ref{Th:W} and Corollary \ref{Cor:W}(1) are workable for
$$\nu(\d z)=\left(|z|^{-1-\alpha}\I_{\{0<z\le 1\}} + |z|^{-1-\alpha_1}\I_{\{z>
1\}}\right) \,\d z$$ with $\alpha\in (0,2)$ and $\alpha_1>1$.

(2) We mention that, by the remarks below \cite[Example 2.18]{LYZ}, one can easily give examples such that the assumptions of Theorem \ref{Th:W} (or Theorem \ref{Th:W1})  are satisfied, but for any $x>0$, $\Pp^x(\tau_0<\infty)>0$ (or even =1), where $\tau_0=\inf\{t>0:X_t=0\}$. Therefore, under assumptions of Theorem $\ref{Th:W}$ (or Theorem \ref{Th:W1})  the invariant probability measure of the process $(X_t)_{t\ge0}$ could be allowed to have an atom at $\{0\}.$
 \end{remark}

The following assertion is furthermore concerned on the strong ergodicity of the process $(X_t)_{t\ge0}$.
\begin{theorem}\label{Cor:var1} Under assumptions of Theorem $\ref{Th:W1}$, if \eqref{e:drift} is strengthened into the condition that there are a constant  $l_0\ge0$ and two nonnegative functions $\Phi_1$ and $\Phi_2$ such that
\begin{equation}\label{e:drift1}\gamma_0(x)-\gamma_0(y)\le \begin{cases}\Phi_1(x-y),\quad &0\le x-y\le l_0,\\
-\Phi_2(x-y),\quad & x-y>l_0,
\end{cases}\end{equation} where $\Phi_1$ is the same  as that in Theorem $\ref{Th:W1}$, and  $\Phi_2\in C^2[l_0,\infty)$ satisfies $\Phi'_2\ge0$ and $\Phi_2''\ge0$ on $[l_0,\infty)$, as well as $$\int_{l_0}^\infty \frac{1}{\Phi_2(s)}\,\d s<\infty,$$ then the process $(X_t)_{t\ge0}$ is strongly ergodic, i.e., there exist the unique
invariant probability measure $\mu$ and constants $C,\lambda>0$  such that for all $t>0$ and
$x\ge 0$,
$$\|\delta_x P_t-\mu\|_{\rm Var}\le Ce^{-\lambda t}.$$   \end{theorem}

Note that since $\Phi_2''\ge0$ on $[2l_0,\infty)$, $\Phi_2(r)\ge \Phi_2(2l_0)+\Phi_2'(2l_0) r.$ So, \eqref{e:drift1} is stronger than \eqref{e:drift} (by choosing $l_0>0$ large enough if necessarily).
A typical example for the function $\Phi_2$ in Theorem \ref{Cor:var1} is that $\Phi_2(r)=c_0r^\delta$ with $c_0>0$ and $\delta>1$.

\ \

We close this section with the following examples on the coefficient $\gamma_0(x)$.

\begin{example}\label{Exm}\begin{itemize}
\item [(1)] Let $\gamma_0(x)=b_1x\log(1+1/x)-b_2x$ with  $b_1,b_2>0$. Then, \eqref{e:drift} holds with $\Phi_1(r)=b_1r\log (1+1/r)$ and $k_2=b_2/2$ for some $l_0>0$ large enough.

\item [(2)] Let $\gamma_0(x)=b_1x-b_2x^\delta$ with $\delta>1$ and $b_1,b_2>0$. Then, \eqref{e:drift1} holds with $\Phi_1(r)=b_1r$ and $\Phi_2(r)=b_2r^\delta/2$ for some $l_0>0$.

\item [(3)] Let  $\gamma_0(x)=b_1x-b_2e^{cx^\delta}$ with $c, \delta,b_1,b_2>0$. Then, \eqref{e:drift1} holds with $\Phi_1(r)=b_1r$, and $\Phi_2(r)=cr^\theta$ with any $\theta>1$ and some $l_0,c>0$.
\end{itemize}
\end{example}

\section{Proofs}

\subsection{Lemmas}

To prove the main results in this paper, we need the following elementary lemmas.

\begin{lemma}\label{lem-test-funct}
For fixed $l_0>0$, let $g\in C[0,2l_0]\cap C^3(0,2l_0]$ be satisfying $g(0)=0$ and
  \begin{equation}\label{lem-test-funct.0}
  g'(r)\geq 0,\ g''(r)\leq 0 \mbox{ and } g'''(r)\geq 0\quad \mbox{for any } r\in (0,2l_0].
  \end{equation}
Then for all $c_1, c_2>0$ the function
  \begin{equation}\label{lem-test-funct.1}
\psi(r)= \begin{cases}
  c_1 r+ \int_0^r e^{-c_2 g(s)}\, \d s ,& r\in [0, 2l_0],\\
 \psi(2l_0)+ \frac{\psi'(2l_0)}{2}\int_0^{r-2l_0}\Big[1+\exp\big(\frac{2\psi''(2l_0)}{\psi'(2l_0)}s\big)\Big]\,\d s, & r\in(2l_0,\infty)
  \end{cases}
  \end{equation}
satisfies
\begin{itemize}
\item[\rm (1)] $\psi\in C^2(\R_+)$ such that $\psi'> 0$ and $\psi''< 0$ on $\R_+$;
\item[\rm (2)]  $\psi'''\geq 0$ and $\psi^{(4)}\leq 0$ on $(0,2l_0]$. In particular, for any $0\le\delta\le r\le l_0$,
$$
\psi(r+\delta)+\psi(r-\delta)-2\psi(r)\le \psi''(r)\delta^2;
$$
\item[\rm (3)] for all $r>0$, $$\min\Big\{c_1, \frac{\psi(2l_0)}{4l_0}, \frac{\psi'(2l_0)}{4}\Big\}r\le \psi(r)\le (1+c_1)r.$$
\end{itemize}
\end{lemma}
\begin{proof} The assertion (1) follows from the definition of $\psi$. The assertion (2) has been proven in  \cite[Lemma 4.1]{LW18}. Since $\|\psi'\|_\infty= 1+c_1$ and $\psi(0)=0$, the second inequality in the assertion (3) holds. On the other hand, for any $r\in [0,2l_0]$, $\psi(r)\ge c_1r$; for any $r\in [4l_0,\infty)$, $\psi(r)\ge \frac{\psi'(2l_0)}{2}(r-2l_0)\ge  \frac{\psi'(2l_0)}{4}r$; for any $r\in (2l_0,4l_0]$,
$\psi(r)\ge \psi(2l_0)\ge \frac{\psi(2l_0)}{4l_0} r$. Combining with all the estimates above, we can prove the first inequality in the assertion (3).  \end{proof}

We have the following typical choice of functions $g$ in the definition \eqref{lem-test-funct.1} for $\psi$.

\begin{lemma}\label{Lem-derivatives}
 For fixed $l_0>0$, let $\Phi_1\in C[0, 2l_0]\cap C^3(0, 2l_0]$ be a nonnegative function such that $\Phi_1(0)=0$, $\Phi_1'\ge0$, $\Phi_1''\le0$ and $\Phi_1'''\ge0$ on $(0,2l_0]$. Suppose that for some $\theta\in (0,1]$, $$\int^r_0\Phi_1(z)z^{\theta-2}\,\d z<\infty,\quad r\in [0,2l_0].$$ For any $c_0>0$, set $$g(r):=r^\theta+c_0\int^r_0\Phi_1(z)z^{\theta-2}\,\d z.$$  Then $g\in C[0,2l_0]\cap C^3(0,2l_0]$ such that $g(0)=0$, and \eqref{lem-test-funct.0} holds;
moreover,
  $$\sup_{0<r\le 2l_0} \left(rg'(r)-\frac{rg''(r)}{g'(r)}\right)<\infty.$$
\end{lemma}
\begin{proof}
Let $g_1(r)=r^\theta$ and $g_2(r)=c_0\int^r_0\Phi_1(z)z^{\theta-2}\,\d z$ for some $\theta\in(0,1]$ and $c_0>0$. It is clear that $g_1(0)=0$, and $g_1$ satisfies \eqref{lem-test-funct.0}. We next claim that $g_2$ also enjoys the property \eqref{lem-test-funct.0}. Indeed, by assumptions, it is clear that $g_2(0)=0$, $g_2'(r)= c_0 \Phi_1(r)r^{\theta-2}\ge0$, and
\begin{align*}g_2''(r)=&c_0(\theta-2)\Phi_1(r)r^{\theta-3}+c_0\Phi_1'(r)r^{\theta-2}=c_0r^{\theta-3}\left((\theta-2)\Phi_1(r)+\Phi_1'(r) r\right)\\
\le&c_0r^{\theta-3}\left(-\Phi_1(r)+\Phi_1'(r) r\right)\le 0 \end{align*} for $r\in (0,2l_0]$,  where in the first inequality  we used the facts that $\theta\in(0,1]$ and $\Phi_1(r)\ge0$ for all $r\in (0,2l_0]$, and the last inequality follows from the facts that $\Phi_1(0)=0$ and $\Phi_1''(r)\le 0 $ for $r\in (0,2l_0]$. Furthermore, for $r\in (0,2l_0]$, \begin{align*}g_2'''(r)=&c_0\Phi_1''(r)r^{\theta-2}+2c_0(\theta-2)\Phi'_1(r)r^{\theta-3}+c_0(\theta-2)(\theta-3)\Phi_1(r)r^{\theta-4}\\
=&c_0r^{\theta-4}\left[(\theta-2)(\theta-3)\Phi_1(r)+2(\theta-2)\Phi'_1(r)r+\Phi_1''(r)r^2\right].\end{align*} By the facts that $\Phi_1(0)=0$ and $\Phi_1'''\ge0$ on $(0,2l_0]$, and the mean value theorem,
\begin{align*}0=&2(2-\theta)\Phi_1(0)\le 2(2-\theta)\Phi_1(r)-2(2-\theta)\Phi_1'(r)r+(2-\theta)\Phi_1''(r)r^2\\
\le&(3-\theta)(2-\theta)\Phi_1(r)-2(2-\theta)\Phi_1'(r)r+\Phi_1''(r) r^2\end{align*} for all $r\in (0,2l_0]$, where in the second inequality above we used the facts that $\theta\in(0,1]$, $\Phi_1\ge0$ and $\Phi_1''\le0$ on $(0,2l_0]$. This implies that $g_2'''\ge0$ on $(0,2l_0]$. Combining with all the estimates above, we prove the desired assertion for $g_2$. Since $g=g_1+g_2$, we show that $g$ satisfies \eqref{lem-test-funct.0}.

Since $g(0)=0$ and $g''(r)\le0$ for all $r\in (0,2l_0]$, $rg'(r)\le g(r)$ for all $r\in (0,2l_0]$ and so $$\sup_{r\in (0,2l_0]} rg'(r)\le \sup_{r\in (0,2l_0]}  g(r)= g(2l_0).$$ On the other hand,
$$g'(r)=\theta r^{\theta-1}+c_0\Phi_1(r)r^{\theta-2}$$ and
\begin{align*}-rg''(r)&=\theta(1-\theta)r^{\theta-1}+c_0(2-\theta)r^{\theta-2}\Phi_1(r)-c_0r^{\theta-1}\Phi'_1(r)\\
&\le\theta(1-\theta)r^{\theta-1}+c_0(2-\theta)r^{\theta-2}\Phi_1(r),\end{align*} where we used the fact that $\Phi_1'\ge0$ on $(0,2l_0]$. Thus, by $\theta\in(0,1]$,
$$\sup_{r\in (0,2l_0]}\frac{-rg''(r)}{g'(r)}\le \sup_{r\in (0,2l_0]}\frac{\theta(1-\theta)r^{\theta-1}+c_0(2-\theta)r^{\theta-2}\Phi_1(r)}{\theta r^{\theta-1}+c_0\Phi_1(r)r^{\theta-2}}\le 2-\theta.$$ Therefore, the proof is complete.
\end{proof}

In the following, for any $f\in C^2(\R_+)$, $\tilde L f(x-y):=\tilde L  F(x,y)$, where $F(x,y)=f(x-y)$.

\begin{lemma}\label{lem-trans}
For any $n\ge 1$, let $\psi_n \in C^2(\R_+)$ be  satisfying $\psi_n(0)=0$ and
$$
\tilde{L}\psi_n(x-y)\le -\lambda \psi_n(x-y),\quad 1/n\le x-y\le n,
$$ where $\lambda>0$ is independent of $n$, $x$ and $y$.
Then for any $t>0$ and $x, y\in \R_+$,
$$
W_{\psi} (P_t(x,\cdot), P_t(y,\cdot))\le \psi(|x-y|)e^{-\lambda t},
$$
where $\psi:=\liminf_{n\to\infty}\psi_n.$
\end{lemma}
\begin{proof}
This lemma follows from the  preserving order property for the coupling process $(X_t,Y_t)_{t\ge0}$ associated with the coupling operator $\tilde L$ proved in Corollary \ref{L:order}, and the arguments in part (2) of  the proof for \cite[Theorem 3.1]{LW18}. So, we omit the details here.
\end{proof}

\subsection{Proofs of the main results}

Now, we are in a position to prove Theorems $\ref{Th:W}$ and \ref{Th:W1}.

\begin{proof}[Proof of Theorem $\ref{Th:W}$] (1) We first verify the assertion when (A2) is satisfied. At the beginning we will prove it under the assumption that \begin{equation}\label{e:sspp1}
\mu_z(\R_+)\ge C_*z^{-\alpha},\quad z\in (0,\kappa].\end{equation} In this case, \eqref{e:jump} is reduced into
$$\gamma_2(x) \ge k_3(x-y)^{\beta},\quad 0\le x-y\le l_0,$$ which is equivalent to  \begin{equation}\label{e:sspp}\gamma_2(x) \ge k_3x^{\beta},\quad 0\le x\le l_0.\end{equation} Throughout the proof, without loss of generality we may and can assume that $l_0\ge1$ and $\kappa\in (0,1]$.
According to the definition \eqref{generator} of the coupling operator $\tilde L$, we
know that for any $f\in C^2(\R_+)$ and any $x>y\ge0$,
\begin{align*}\tilde L f(x-y)=&\frac{1}{2} \gamma_2(y)\big[f ((x-y)+(x-y)_\kappa
)+f((x-y)-(x-y)_\kappa)-2f(x-y)\big]\\
&\qquad\quad\times\mu_{(x-y)_\kappa}(\R_+)\\
&+\left(\gamma_2(x)-\gamma_2(y)\right) \int_0^\infty
(f(x-y+z)-f(x-y)-f'(x-y)z)\,\nu(\d z)\\
&+\left(\gamma_0(x)-\gamma_0(y)\right)f'(x-y)+\frac{1}{2}(\sqrt{\gamma_1(x)}+\sqrt{\gamma_1(y)})^2f''(x-y).\end{align*}

In the following, we will take $f$ to be the function $\psi$ defined
by \eqref{lem-test-funct.1}, where the constants $c_1,c_2>0$ and the
function $g$ are determined later.
According to Lemma \ref{lem-test-funct}(1), $\psi''(r)\le0$ for all $r>0$, and so by the mean value theorem, for any $x>y\ge0$ and $z>0$,
$$\psi(x-y+z)-\psi(x-y)-\psi'(x-y)z\le 0;$$ moreover, thanks to Lemma \ref{lem-test-funct}(2),
for $0<x-y\le l_0 $ and $0<z\le l_0$, we see
$$\psi(x-y+z)-\psi(x-y)-\psi'(x-y)z\le \frac{1}{2}\psi''(x-y)z^2+\frac{1}{6}\psi'''(x-y)z^3$$
and
$$ \psi ((x-y)+(x-y)_\kappa
)+\psi((x-y)-(x-y)_\kappa)-2\psi(x-y)\le \psi''(x-y)(x-y)_\kappa^2.$$
Then, by using all the estimates above,
\eqref{e:sspp1}, the increasing property of $\gamma_2(x)$ and \eqref{e:drift},  for any $c_0\in (0,1/l_0]\subset (0,1]$ (whose value will be chosen later), and any $x>y\ge0$ with
$0<x-y\le l_0$,
\begin{align*} \tilde L \psi(x-y)\le &
\frac{C_*}{2}\gamma_2(y)\psi''(x-y)(x-y)_\kappa^{2-\alpha}\\
&+(\gamma_2(x)-\gamma_2(y))\int_0^{c_0(x-y)} \left(\psi(x-y+z)-\psi(x-y)-\psi'(x-y)z\right)\,\nu(\d z)\\
&+\Phi_1(x-y)\psi'(x-y)+\frac{1}{2}(\sqrt{\gamma_1(x)}+\sqrt{\gamma_1(y)})^2\psi''(x-y)\\
\le &\frac{C_*}{2}\gamma_2(y)\psi''(x-y)(x-y)_\kappa^{2-\alpha}\\
&+(\gamma_2(x)-\gamma_2(y))\psi''(x-y)\cdot \frac{1}{2}\int_0^{c_0(x-y)} z^2\,\nu(\d z)\\
&+\frac{ c_0(x-y)(\gamma_2(x)-\gamma_2(y))}{3}\psi'''(x-y) \cdot
\frac{1}{2}\int_0^{c_0(x-y)} z^2\,\nu(\d z)\\
&+\Phi_1(x-y)\psi'(x-y)\\
\le &-\frac{C_*c_2}{2}\frac{\kappa^{2-\alpha}}{l_0^{2-\alpha}}\gamma_2(y)g'(x-y)e^{-c_2g(x-y)}(x-y)^{2-\alpha}\\
&-\frac{c_2g'(x-y)e^{-c_2g(x-y)}}{2}\left(\int_0^{c_0(x-y)} z^2\,\nu(\d z)\right)(\gamma_2(x)-\gamma_2(y))\\
&\qquad\times\left[1-\frac{c_0(x-y)}{3}\left(c_2g'(x-y)-\frac{g''(x-y)}{g'(x-y)}\right)\right]\\
&+\Phi_1(x-y)(c_1+e^{-c_2g(x-y)}),
\end{align*} where in the second inequality we used the fact that $\psi'''(r)\ge0$ for all $r\in (0,2l_0]$.

Next, we choose $$g(r)=r^{\alpha-\beta}+c_3\int^r_0\Phi_1(s)s^{\alpha-\beta-2}\,\d s$$ and
   $$c_0=\min\left\{\frac{1}{l_0},\frac {1}{\sup_{0<r\le l_0} \frac{-rg''(r)}{g'(r)}}\right\},\quad c_2=\frac{\sup_{0<r\le l_0} \frac{-rg''(r)}{g'(r)}}{\sup_{0<r\le l_0} r {g'(r)}},\quad  c_1=e^{-c_2g(l_0)}.$$ Note that, by assumptions, we know that $g(r)$ is well defined. Since $0<\alpha-\beta\le 1$, according to Lemma \ref{Lem-derivatives}, we know that $c_0, c_2\in (0,\infty)$. Then, due to \eqref{e:sspp},
\begin{align*} \tilde L \psi(x-y)\le &
 -\frac{C_*c_2}{2}\frac{\kappa^{2-\alpha}}{l_0^{2-\alpha}}\gamma_2(y)g'(x-y)e^{-c_2g(x-y)}(x-y)^{2-\alpha}\\
&-\frac{C_*c_2c_0^{2-\alpha}}{6} (\gamma_2(x)-\gamma_2(y))g'(x-y)e^{-c_2g(x-y)}(x-y)^{2-\alpha}\\
&+2\Phi_1(x-y)e^{-c_2g(x-y)}\\
\le& - \frac{C_*c_2}{2}\min\left\{\frac{\kappa^{2-\alpha}}{l_0^{2-\alpha}}, \frac{c_0^{2-\alpha}}{3} \right\}\gamma_2(x)g'(x-y)e^{-c_2g(x-y)}(x-y)^{2-\alpha}\\
&+2\Phi_1(x-y)e^{-c_2g(x-y)}\\
\le& - \frac{C_*c_2k_3}{2}\min\left\{\frac{\kappa^{2-\alpha}}{l_0^{2-\alpha}}, \frac{c_0^{2-\alpha}}{3} \right\} g'(x-y)e^{-c_2g(x-y)}(x-y)^{2+\beta-\alpha}\\
&+2\Phi_1(x-y)e^{-c_2g(x-y)}.
\end{align*}
Furthermore, taking $$c_3=2\left[\frac{C_*c_2k_3}{2}\min\left\{\frac{\kappa^{2-\alpha}}{l_0^{2-\alpha}}, \frac{c_0^{2-\alpha}}{3} \right\}\right]^{-1}$$
and recalling
$$g'(r)=(\alpha-\beta)r^{\alpha-\beta-1}+c_3\Phi_1(r)r^{\alpha-\beta-2},$$ we arrive at that for any $x,y\in \R_+$ with $0<x-y\le l_0$,
 \begin{equation}\label{test-1}\begin{split}  \tilde L \psi(x-y)
 &\le
- \frac{C_*c_2k_3(\alpha-\beta)}{2}\min\left\{\frac{\kappa^{2-\alpha}}{l_0^{2-\alpha}}, \frac{c_0^{2-\alpha}}{3} \right\}  e^{-c_2g(x-y)}(x-y) \\
&\le - \frac{C_*c_2k_3(\alpha-\beta)}{2}\min\left\{\frac{\kappa^{2-\alpha}}{l_0^{2-\alpha}}, \frac{c_0^{2-\alpha}}{3} \right\}  e^{-c_2g(l_0)}(x-y) .\end{split}
\end{equation}

On the other hand, for any $x-y>l_0$, according to  all the estimates above  for the function $\psi$ and \eqref{e:drift},
 \begin{equation}\label{test-2}\tilde L \psi(x-y)\le -k_2(x-y)\psi'(x-y)\le -\frac{k_2\psi'(2l_0)}{2}(x-y),\end{equation} where in the last inequality we used the fact that $\psi''\le0$ and the definition of $\psi$.

According to \eqref{test-1}, \eqref{test-2} and Lemma \ref{lem-test-funct}(3), we know that
for any  $0<y<x$,
$$ \tilde L \psi(x-y)\le -\lambda\psi(x-y),$$ where
$$\lambda=\frac{1}{1+c_1}\min\left\{\frac{C_*c_2k_3(\alpha-\beta)}{2}\min\left\{\frac{\kappa^{2-\alpha}}{l_0^{2-\alpha}}, \frac{c_0^{2-\alpha}}{3} \right\}  e^{-c_2g(l_0)}, \frac{k_2\psi'(2l_0)}{2}\right\}.$$
This along with Lemma \ref{lem-trans} yields that for any $t>0$ and $x, y\in \R_+$,
$$
W_{\psi} (P_t(x,\cdot), P_t(y,\cdot))\le \psi(|x-y|)e^{-\lambda t}.
$$ Hence, the required assertion follows from the inequality above and Lemma \ref{lem-test-funct}(3).

When \begin{equation}\label{e:sspp1}\gamma_2(x)-\gamma_2(y) \ge k_3(x-y)^{\beta},\quad 0\le x-y\le l_0,\end{equation} one can follow the arguments above to obtain the desired assertion. Indeed, in this case we can get rid of the term involved $\mu_{(x-y)_\kappa}(\R_+)$ in estimates for  $\tilde L \psi(x-y)$ for  any $x,y\in \R_+$ with $0<x-y\le l_0$, since this term is non-positive. We also note that under \eqref{e:sspp1} we can also directly apply the coupling operator $L^*$ and the associated coupling process mentioned in Remark \ref{remark2.4}; however, such coupling can not deal with  the case that \eqref{e:sspp} is satisfied. This explains the reason why we adopt the refined basic coupling for the non-local part of the operator $L$, rather than simply apply the synchronous coupling.

(2) We next verify the assertion when (A1) is satisfied. Let $\psi$
be the function defined by \eqref{lem-test-funct.1}.  Then, we get
from estimates for $\psi$, \eqref{e:diff}
and \eqref{e:drift} that,  for any $x>y$ with $0<x-y\le l_0$,
\begin{align*}\tilde L \psi(x-y)\le & (\gamma_0(x)-\gamma_0(y))\psi'(x-y)+\frac{k_3}{2}(x-y)^\beta\psi''(x-y)\\
\le & \Phi_1(x-y)\psi'(x-y)+ \frac{k_3}{2}(x-y)^\beta\psi''(x-y),
\end{align*}  where in the first inequality we used the fact that
$$(\sqrt{\gamma_1(x)}+\sqrt{\gamma_1(y)})^2\ge k_3(x-y)^\beta,\quad 0\le x-y\le l_0,$$ thanks to \eqref{e:diff}. Furthermore, we choose $$g(r)=r^{2-\beta}+c_3\int^r_0\Phi_1(s)s^{-\beta}\,\d s.$$ Similarly, with possible choice of constants $c_1$, $c_2$ and $c_3$
in the definition of $\psi$, one can follow the argument in part (1) to verify
the desired assertion. The details are omitted here.
\end{proof}

\begin{proof}[Proof of Theorem $\ref{Th:W1}$] For simplicity, we only verify the case that (A2) is satisfied and that $
\mu_z(\R_+)\ge C_*z^{-\alpha}$ for all $z\in (0,\kappa]$ (i.e., \eqref{e:sspp} holds), since one can prove the desired assertion similarly (and even easier) for other cases.

Without loss of generality, we assume that $l_0\ge1$ and $\kappa\in (0,1]$.
  For any $n\ge1$, define $f_n\in C^2(\R_+)$ such that $$f_n(r)\begin{cases} =\psi(r),&\quad 0<r\le 1/(n+1),\\
\le 1+b\left(\frac{r}{1+r}\right)^\theta+\psi(r),&\quad 1/(n+1)<r\le 1/n,\\
=1+b\left(\frac{r}{1+r}\right)^\theta+\psi(r),&\quad r\ge 1/n,\end{cases}$$ where $b>0$ chosen later, $\theta=(\alpha-\beta)/2\in (0,1)$, and $\psi$ is defined by \eqref{lem-test-funct.1} (which is the one in part (1) of the proof of Theorem $\ref{Th:W}$ with some modification on the associated constant $c_3$).
We will verify that there exists a constant $\lambda>0$ such that for any $n\ge1$ and $x-y>1/n$, \begin{equation}\label{e:total}\tilde L f_n(x-y)\le - \lambda
f_n(x-y).\end{equation}
If \eqref{e:total} holds, then, the assertion follows from Lemma \ref{lem-trans} and the fact that $$\liminf_{n\to\infty} f_n(x,y)=\I_{\{x\neq y\}}(1+b(|x-y|/(1+|x-y|))^\theta+\psi(|x-y|)\asymp \I_{\{x\neq y\}}(1+ |x-y|).$$ Here  and in what follows, for any nonnegative functions $f,g$, $f\asymp g$ means that there is a constant $c\ge 1$ such that $c^{-1} f(x)\le g(x)\le c f(x)$ on the region of $x$.

In the following, let $\psi_0(r)=b (r/(1+r))^\theta$, and $l^*_0\in (0,\kappa]$ determined later. Then, by \eqref{e:drift}, for $n\ge1$ and $1/n\le l_0^*\le x-y\le l_0$,
\begin{align*}\tilde L \psi_0(x-y)\le &\Phi_1(x-y)\psi_0'(x-y)\le b\theta\Phi_1(x-y)(1+l_0^*)^{-2}{(l_0^*/(1+l_0^*))}^{\theta-1},
\end{align*} where in the first inequality  we used the fact that $\psi_0''\le 0$ and the definition of the coupling operator $\tilde L$ given by \eqref{generator}. On the other hand, with the same function $\psi$ (with the same function $g$ and the constants $c_0,c_1,c_2)$ in  the proof of Theorem $\ref{Th:W}$, we find that for $n\ge1$ and $1/n\le l_0^*\le x-y\le l_0$,
\begin{align*}\tilde L \psi(x-y) \le &- \frac{C_*c_2k_3}{2}\min\left\{\frac{\kappa^{2-\alpha}}{l_0^{2-\alpha}}, \frac{c_0^{2-\alpha}}{3} \right\} g'(x-y)e^{-c_2g(x-y)}(x-y)^{2+\beta-\alpha}\\
&+2\Phi_1(x-y)e^{-c_2g(x-y)}\end{align*}
Combining both estimates together, we obtain that for any $n\ge1$ and $1/n\le l_0^*\le x-y\le l_0$,
\begin{align*}\tilde L f_n(x-y) \le &- \frac{C_*c_2k_3}{2}\min\left\{\frac{\kappa^{2-\alpha}}{l_0^{2-\alpha}}, \frac{c_0^{2-\alpha}}{3} \right\} g'(x-y)e^{-c_2g(x-y)}(x-y)^{2+\beta-\alpha}\\
&+ \Phi_1(x-y)e^{-c_2g(x-y)} (2+2b\theta e^{c_2g(l_0)} {l_0^*}^{\theta-1}),\end{align*} where in the inequality above we used the fact that $l_0^*\in (0,1)$. Furthermore, choosing $b=2^{-1}e^{-c_2g(l_0)}$ and noticing that $\theta\in(0,1]$, we arrive at
\begin{align*}\tilde L f_n(x-y) \le &- \frac{C_*c_2k_3}{2}\min\left\{\frac{\kappa^{2-\alpha}}{l_0^{2-\alpha}}, \frac{c_0^{2-\alpha}}{3} \right\} g'(x-y)e^{-c_2g(x-y)}(x-y)^{2+\beta-\alpha}\\
&+ \Phi_1(x-y)e^{-c_2g(x-y)} (2+{l_0^*}^{\theta-1}). \end{align*}  Now, replacing the constant $c_3$ in the proof for Theorem $\ref{Th:W}$ by
$$c_3=(2+{l_0^*}^{\theta-1})\left[\frac{C_*c_2k_3}{2}\min\left\{\frac{\kappa^{2-\alpha}}{l_0^{2-\alpha}}, \frac{c_0^{2-\alpha}}{3} \right\}\right]^{-1},$$  we can get that for any $n\ge1$ and $1/n\le l_0^*\le x-y\le l_0$,
$$\tilde L f_n(x-y)\le - \frac{C_*c_2k_3(\alpha-\beta)}{2}\min\left\{\frac{\kappa^{2-\alpha}}{l_0^{2-\alpha}}, \frac{c_0^{2-\alpha}}{3} \right\}  e^{-c_2g(l_0)}(x-y) .$$

Next, we  turn to the case that $1/n<x-y\le l_0^*$ with any $n\ge 1$. According to the definition of the function $\psi$ and the proof of Theorem $\ref{Th:W}$, we know that \eqref{test-1} still holds true for any $1/n<x-y\le l_0^*$. On the other hand, by the fact that $\psi_0''\le 0$, for any $1/n<x-y\le l_0^*\le \kappa$,
\begin{align*}\tilde L \psi_0(x-y)\le &\frac{1}{2} \gamma_2(y)\Big[\psi_0 ((x-y)+(x-y)_\kappa
)+\psi_0((x-y)-(x-y)_\kappa)-2\psi_0(x-y)]\\
&\qquad\quad\times\mu_{x-y}(\R_+)\\
&+\left(\gamma_2(x)-\gamma_2(y)\right) \int_0^{x-y}
(\psi_0(x-y+z)-\psi_0(x-y)-\psi_0'(x-y)z)\,\nu(\d z)\\
&+\left(\gamma_0(x)-\gamma_0(y)\right)\psi_0'(x-y)\\
\le &\frac{C_*}{2}\gamma_2(y)\psi_0''(x-y)(x-y)^{2-\alpha}+\frac{1}{2}\left(\gamma_2(x)-\gamma_2(y)\right) \psi_0''(2(x-y))\int_0^{x-y}z^2\,\nu(\d z)\\
&+\left(\gamma_0(x)-\gamma_0(y)\right)\psi_0'(x-y)\\
\le& -b\theta C_* k_3(1-\theta)2^{\theta-2}(1+2l_0^*)^{-\theta-2} (x-y)^{\beta-\alpha+\theta} +b\theta \Phi_1(x-y)(x-y)^{\theta-1}\\
\le& b\theta(x-y)^{\beta-\alpha+\theta}\left[-\frac{C_* k_3 (1-\theta) }{108}+  \sup_{0<r\le l^*_0}(\Phi_1(r)r^{\alpha-\beta-1})\right]. \end{align*} Here, in the second inequality we used the facts that $l_0^*\le \kappa$,
$$\psi_0(r+s)+\psi_0(r-s)-2\psi_0(r)s\le \psi_0''(r) s^2$$ and
$$\psi_0(r+s)-\psi_0(r)-\psi_0'(r) s\le \frac{s^2}{2}\psi''_0(2r)$$ for any $ 0<s\le r;$  the third inequality follows from \eqref{e:sspp} and  \eqref{e:drift};
and the fourth inequality is deduced from the fact that $\theta, l_0^*\in (0,1)$. Since $\limsup_{r\to0} \Phi_1(r)r^{\alpha-\beta-1}=0$, we can find $l_0^*\in (0,1]$ small enough such that for all $n\ge1$ and  $1/n<x-y\le l_0^*\le 1$,
$$\tilde L \psi_0(x-y) \le -\frac{ b\theta C_* k_3(1-\theta) }{216} (x-y)^{\beta-\alpha+\theta}.$$
Hence, for any $n\ge1$ and $1/n<x-y\le l_0^*$,
$$\tilde L f_n(x-y)\le -\frac{ b\theta C_* k_3 (1-\theta)}{216} {l_0^*}^{-(\alpha-\beta)/2}.$$

Finally, according to \eqref{test-2} and the facts that $\psi_0'\ge0$ and $\psi_0''\le0$, we find that for any $x-y>l_0$,
$$\tilde L f_n(x-y)\le -k_2(x-y)\psi'(x-y)\le -\frac{k_2\psi'(2l_0)}{2}(x-y).$$

Combining all the estimates above for $\tilde L f_n(x-y)$, we can obtain \eqref{e:total}, thanks to the fact that $f_n(r) \asymp (1+r)$ for all $r\neq0$. The proof is complete.   \end{proof}

\begin{proof}[Proof of Corollary $\ref{Cor:W}$]
By \cite[Proposition 2.3]{FL10}, we see that under assumptions of Theorem \ref{Th:W} (or Theorem \ref{Th:W1}), there exist constants $C_1, K>0$ such that for all $x\in \R_+$ and $t>0$,
$$\Ee^{x}  X_t \le C_1(1+x)e^{Kt}.$$ It then follows that $\delta_xP_t\in\mcr{P}_1$ and hence
$\mu P_t\in\mcr{P}_1$ for each $\mu\in\mcr{P}_1$, where $\mcr{P}_1$ is the space of all probability measures on $(\R_+,\mcr{B}(\R_+))$ with the first finite moment.
With this at hand, the proof of Corollary $\ref{Cor:W}$ essentially follows from that of \cite[Corollary 1.8]{LW16}. We omit the details here.
 \end{proof}

\begin{proof}[Proof of Theorem $\ref{Cor:var1}$] Similar to the proof of Theorem $\ref{Th:W1}$, we only verify the case that (A2) is satisfied and that $
\mu_z(\R_+)\ge C_*z^{-\alpha}$ for all $z\in (0,\kappa]$ with $C_*,\kappa>0$ and $\alpha\in (0,2)$ (i.e., \eqref{e:sspp} holds). To verify the desired assertion, we will use the following test function
\begin{equation}\label{lem-test-funct.222}
\psi(r)= \!\begin{cases}
  c_1 r+ \int_0^r e^{-c_2 g(s)}\, \d s ,& r\in [0, 2l_0],\\
 \psi(2l_0)+ A\int_0^{r-2l_0}\frac{1}{\Phi_2(Bs+2l_0)}\,\d s+ \delta A\int_0^{r-2l_0}\frac{1}{\Phi_2(s+2l_0)}\,\d s, & r\in(2l_0,\infty),
  \end{cases}
  \end{equation} where $A=\frac{\psi'(2l_0)\Phi_2(2l_0)}{\delta+1}$, $B=-\frac{\psi''(2l_0)\Phi_2(2l_0)(\delta+1)}{\psi'(2l_0)\Phi_2'(2l_0)}-\delta$ and $\delta>0$ is sufficient small such that $B>0$. Note that the modification between the test function $\psi$ given by \eqref{lem-test-funct.222} and the one in the the proof of Theorem \ref{Th:W} is only made for $r\in (2l_0,\infty)$. It is easy to see that $\psi\in C^2(\R_+)$ such that  $\psi'>0$ and $\psi''<0$ on $\R_+$; moreover, by $\int_1^\infty \frac{1}{\Phi_2(s)}\,\d s<\infty,$ $\psi\in C_b(\R_+)$.

  With the test function $\psi$ above, we can define a sequence of functions $\{f_n\}_{n\ge1} \subset C^2(\R_+)$ such that
  $$f_n(r)\begin{cases} =\psi(r),&\quad 0<r\le 1/(n+1),\\
\le 1+b\left(\frac{r}{1+r}\right)^\theta+\psi(r),&\quad 1/(n+1)<r\le 1/n,\\
=1+b\left(\frac{r}{1+r}\right)^\theta+\psi(r),&\quad r\ge 1/n, \end{cases}$$ where $b$ and $\theta$ are the same as in the proof of Theorem \ref{Th:W1}. In particular, $\{f_n\}_{n\ge1}$ is
  uniformly  bounded, i.e.\ $\sup_{n\ge1}\|f_n\|_\infty<\infty$.

  Following the proof of Theorem \ref{Th:W1}, we have the same estimate for $\tilde L f_n(x-y)$ when $1/n\le x-y\le l_0$. Next, we consider the estimate for $x-y>l_0$. First, by the facts that $\psi_0'\ge0$, $\psi_0''\le0$ and $\psi''\le0$, for any $x-y\ge 2l_0$, we have
$$\tilde L f_n(x-y) \le -\Phi_2(x-y) \psi'(x-y) \le -\delta A.
$$ On the other hand, for any $l_0<x-y\le 2l_0$,
$$\tilde L f_n(x-y) \le -\Phi_2(x-y) \psi'(x-y) \le -c_1\Phi_2(x-y)\le -c_1\Phi_2(l_0).$$
Combining with all the estimates above, we obtain that there is a constant $\lambda>0$ such that for all $x,y\in \R_+$ with $x>y$ and $n\ge1$,
  $$\tilde L f_n(x,y)\le - \lambda.$$ This along with the fact that $f_n(r)\asymp \I_{(0,\infty)}$ for all $n\ge 1$ and Lemma \ref{lem-trans} in turn yields that there exists a positive constant $C$ so that for all $t>0$ and $x,y\in \R_+$,
$$\|P_t(x,\cdot)-P_t(x,\cdot)\|_{\rm Var}\le C e^{-\lambda t}.$$ Hence, the desired assertion follows from the proof of Corollary \ref{Cor:W}.
    \end{proof}

Next, we turn to the

\begin{proof}[Proof of Theorem $\ref{t1}$] Condition \eqref{e:drift00} means that \eqref{e:drift} holds with $$\Phi_1(r)=k_1r\log \left(\frac{4l_0}{r}\right).$$

When (1) holds, Assumption (A1) in Theorem \ref{Th:W} is satisfied; see the remarks below Theorem \ref{Th:W1}.

When  $\nu(\d z)\ge c_0\I_{\{0<z\le 1\}}z^{-1-\alpha}\,\d z$ for some $c_0>0$ and $\alpha\in (0,2)$, by \cite[Example 1.2]{LW18}, we know that $\mu_z(\R_+)\ge C_*z^{-\alpha}$ for all $z\in (0,\kappa]$ with some $C_*,\kappa>0$. Then, that condition (2) holds implies that Assumption (A2) in Theorem \ref{Th:W} is satisfied too; see also the remarks below Theorem \ref{Th:W1}.

For $\gamma_2(x)$ given in (3), we have for all $x>y \ge0$,
$$\gamma_2(x)-\gamma_2(y)\ge b_2(x^{r_2}-y^{r_2}).$$
Since $r_2\in [1,\alpha)$ with $\alpha\in (1,2)$,  for all $x>y\ge0$,
$$x^{r_2}-y^{r_2}\ge x^{r_2-1} (x-y)\ge (x-y)^{r_2}.$$
Hence,  Assumption (A2) in Theorem \ref{Th:W} is satisfied.
With all the conclusions above, we can obtain the desired assertion from Theorems \ref{Th:W}, \ref{Th:W1} and \ref{Cor:var1}, as well as Corollary \ref{Cor:W}.
  \end{proof}

Finally, we present the

\begin{proof}[Proof of Example $\ref{Exm}$] (1) For any $x>y>0$,
\begin{align*}\gamma_0(x)-\gamma_0(y)=&[b_1x\log (1+1/x)-b_1y\log(1+1/y)]-[b_2x-b_2y]\\
\le& b_1(x-y)\log(1+1/x)-b_2(x-y)\\
\le& b_1(x-y)\log(1+1/(x-y))-b_2(x-y).\end{align*} This implies that  \eqref{e:drift} holds with $\Phi_1(r)=b_1r\log(1+1/r)$ and $k_2=b_2/2$, by setting $l_0>0$ large enough such that $b_1\log(1+l_0^{-1})\le b_2/2$. We note that, by some elementary calculations, $\Phi_1(r)=b_1r\log(1+1/r)$ satisfies all the assumptions in Theorem \ref{Th:W}.

(2) Note that for all $\delta>1$ and $x>y\ge0$,
$$x^\delta-y^\delta\ge x^{\delta-1}(x-y)\ge (x-y)^\delta.$$ Then,
\begin{align*}\gamma_0(x)-\gamma_0(y)=&[b_1x -b_1y ]-[b_2x^\delta-b_2y^\delta]\\
\le& b_1(x-y) -b_2(x-y)^\delta.\end{align*} Hence, we know that \eqref{e:drift1} holds with $\Phi_1(r)=b_1r$,  $\Phi_2(r)=b_2 r^\delta/2$ and $l_0=(2b_1/b_2)^{1/(\delta-1)}.$

(3) We consider the function $x\mapsto e^{cx^\delta}$ with $c,\delta>0$ on $\R_+$. For any $x,y\in \R_+$ with $x-y\ge l_0$ and some $l_0>0$ large enough,
$$e^{cx^\delta}-e^{cy^\delta}\ge \frac{e^{cx^\delta}}{x}(x-y)\ge c_0x^\theta (x-y)\ge c_0(x-y)^{1+\theta},$$ where $c_0$ and $\theta$ are positive constants. Here, in the first inequality we used the fact that the function $x\mapsto \frac{e^{cx^\delta}}{x}$ is increasing for $x>0$ large enough, and the second inequality follows from the fact that $\frac{e^{cx^\delta}}{x} \ge c_0x^\theta $ for $x>0$ large enough, where $\theta>0$ can be chosen to be any positive constant. With aid of this inequality, we can prove the desired assertion by following the argument in (2).  \end{proof}

\noindent \textbf{Acknowledgements.}
The research of Jian Wang is supported by the National Natural Science Foundation of China (No.\ 11831014),
the Program for Probability and Statistics:
Theory and Application (No.\ IRTL1704)
and the Program for Innovative Research Team in Science and Technology
in Fujian Province University (IRTSTFJ).

\end{document}